\documentclass{amsart}
\newcommand{\nbold}{\bm{\mathrm{n}}}
\newcommand{\mbold}{\bm{\mathrm{m}}}
\newcommand{\Ci}{\mathscr{C}}
\newcommand{\Sone}{\mathbb{S}^1}
\newcommand{\Torus}{\mathbb{T}^2}

\newcommand{\R}{\mathbb{R}}

\newcommand{\Z}{\mathbb{Z}}

\newcommand{\T}{\mathbb{T}}

\newcommand{\ve}{\varepsilon}

\newcommand{\vertiii}[1]{{\left\vert\kern-0.25ex\left\vert\kern-0.25ex\left\vert #1 
    \right\vert\kern-0.25ex\right\vert\kern-0.25ex\right\vert}}

\newcommand{\Jac}{\operatorname{Jac}}
\newcommand{\Leb}{\operatorname{Leb}}

 \newcommand{\mathfrcbfslanted}[1]{\text{{\slshape\cursive \textcal{#1}}}}

 \newcommand{\ncal}{\mathfrcbfslanted{n}}

\usepackage{frcursive}
\usepackage{bm}
\usepackage{enumerate}
\usepackage{mathrsfs}
\usepackage{color}
\usepackage{amssymb,amsmath,amsthm,amscd}
\usepackage[T1]{fontenc}
\usepackage{mathtools}

\theoremstyle{plain}
\newtheorem{thm}{Theorem}

\newtheorem{lem}[thm]{Lemma}
\newtheorem{prop}[thm]{Proposition}

\theoremstyle{definition}
\newtheorem{dfn}[thm]{Definition}

\theoremstyle{remark}
\newtheorem{remark}[thm]{Remark}

\begin{document}

\title[Partial captivity]{The partial captivity condition for U(1) extensions of expanding maps on the circle}

\date{\today}

\author{Yushi Nakano}
\address[Yushi Nakano]{Advanced mathematical institute, Osaka City University, Osaka,
558-8585, JAPAN}
\email{yushi.nakano@gmail.com}

\author{Masato Tsujii}
\address[Masato Tsujii]{Department of Mathematics,
Kyushu University, Fukuoka, 812-8581,
Japan}
\email{tsujii@math.kyushu-u.ac.jp}

\author{Jens Wittsten}
\address[Jens Wittsten]{Department of Mathematical Sciences,
Ritsumeikan University,
Kusatsu, 525-8577, Japan}
\email{jensw@maths.lth.se}



\begin{abstract}
This paper concerns the compact group extension 
\[
f:\Torus\to \Torus,\quad 
f  (x,s)=
(E(x), s+\tau(x)\ \text{mod }1)
\]
of an expanding map $E:\Sone\to \Sone$. 
The dynamics of $f$ and its stochastic perturbations have previously been studied under the so-called {\it partial captivity condition}. 
Here we prove a supplementary result that shows that partial captivity is a $\Ci^r$ generic condition on $\tau$, once we fix $E$. 
\end{abstract}

\maketitle

\section{Introduction}\label{section:introduction}

Let $E:\Sone \to \Sone$ 
be a  $\Ci^r$ expanding map on the circle $\Sone=\R/\Z$ of degree $\ell\ge2$.
For each 
$\tau :\Sone \to \R$ in $\Ci^r$, we consider the compact group extension
\begin{equation*}
f=f_{\tau}=f_{E,\tau} :\Torus\to \Torus,\quad 
f  (x,s)=
(E(x), s+\tau(x)\ \text{mod }1)
\end{equation*}
where $\Torus$ denotes the torus $\Sone\times\Sone$.
This is one of the typical examples of partially hyperbolic dynamical systems.  
A naive expectation about the dynamics of $f$ is that we will observe ``virtually'' random dynamics in the fibers, the randomness being driven by the chaotic dynamics of $E$, and consequently that the dynamics of $f$  will be strongly mixing. 
This is of course not true if the function $\tau$ does not transmit the randomness of the dynamics of $E$ to that in the fibers. Indeed, if $\tau(x)\equiv c$, we will observe just a rigid rotation by $c$ in the fibers. Further, if $\tau$ is {\it cohomologous to a constant}, 
i.e.~$\tau(x)=\varphi(E(x))-\varphi(x)+c$ for some $\varphi\in\Ci^r(\Sone)$ and $c\in\R$, 
the dynamics of $f_{E,\tau}$ is conjugated to the case $\tau(x)\equiv c$. To obtain rigorous statements that realize the naive idea described above we must therefore impose some condition on $\tau$. It is of course preferable if such a condition is generic, that is, holds for most systems $f$. 
It is known that $f$ is exponentially mixing once $\tau$ is not cohomologous to a constant. (See \cite[Section 3]{Dolgopyat2002} for instance.) This provides a rather complete description of the mixing property of $f$. 
However, stronger conditions are needed in order to study finer structures of the dynamics, such as spectral properties of the associated transfer operators. 

Below we define and discuss the {\it partial captivity condition} for $f$, which implies roughly that the dynamics of $f$ have properties contrary to those in the case $\tau(x)\equiv c$. (However, we wish to emphasize that the partial captivity condition is strictly stronger than not being cohomologous to a constant, see Appendix \ref{app:B}.) The condition was introduced by Faure~\cite{Faure} to study spectral properties of the associated Perron-Frobenius operators. 
In the previous paper \cite{nw}, the first and third author studied fine properties of stochastic perturbations of $f$ again assuming this condition. 
In this paper, we prove a supplementary result showing that the partial captivity condition is indeed a generic condition. 

Before we state the condition we introduce some notation.  
Note that, by the definition of an expanding map, there are constants $1<\lambda\le \Lambda$ such that
\begin{equation*}
\lambda\le   E' (x)\le \Lambda\quad\text{for all $x\in \Sone$}.
\end{equation*}
Let us set
\begin{equation*}
\vartheta_\tau:=\lVert\tau'\rVert_\infty/(\lambda-1).
\end{equation*}
Fix some $R>\|\tau'\|_\infty$ and put
\begin{equation*}
\vartheta_R:=R/(\lambda-1)>\vartheta_\tau.
\end{equation*}
Then the corresponding cone $\mathscr K_R=\{(\xi,\eta)\in\R^2:\lvert\eta\rvert\le\vartheta_R\lvert\xi\rvert\}$ 
is (forward) invariant under the Jacobian matrix
\[
D f(z)=\left( \begin{array}{cc} E  '(x) & 0\\ \tau'(x) & 1\end{array} \right)
\quad z=(x,s)\in\T^2.
\]
More precisely we have for all $z\in \T^2$ and $m\ge1$ that
\begin{equation}\label{eq:cone_inclusion}
D f^m(z)\mathscr K_R\subset \mathscr K_{R'_m}\subset \mathscr K_R\quad
\text{where $R'_m=\lVert\tau'\rVert_\infty+\lambda^{-m}(R-\lVert\tau'\rVert_\infty)<R$.} 
\end{equation}
 For  $z\in\T^2$ and  $n\ge1$, let us consider  
the images of $\mathscr K_R$ by $Df^n$ in $T_z\Torus$, i.e.
\begin{equation}\label{eq:imagecones}
Df ^n(\zeta) \mathscr K_R \quad \text{for $\zeta\in f^{-n}(z)$}.
\end{equation}
It is not difficult to see that $\tau$ is  cohomologous to a constant if and only if all the cones in \eqref{eq:imagecones} have a line in common at every point $z\in \Torus$ and $n\ge 1$. Thus we naturally come to the idea of considering  transversality between the cones \eqref{eq:imagecones}. 
As a way to quantify this notion, we set
\begin{equation*}
\ncal(\tau,R;n)=\sup_z \sup_v
\# \{ \zeta\in f^{-n}(z)\mid v\in Df^n(\zeta)\mathscr K_R\}
\end{equation*}
where $\sup_v$ denotes the supremum over unit vectors $v\in \mathbb{R}^2$. 
This is sub-mulitiplicative as a function of $n$. Indeed, from \eqref{eq:cone_inclusion}, we have
\begin{equation}\label{eq:ncal_submultplicative}
\ncal(\tau,R;n+m)\le \ncal(\tau,R'_m;n)\cdot \ncal(\tau,R;m)\le \ncal(\tau,R;n)\cdot \ncal(\tau,R;m).
\end{equation}
By Fekete's lemma, the limit
\begin{equation*}
\ncal(\tau):=\lim_{n\to\infty}\ncal(\tau,R;n)^\frac{1}{n}
\end{equation*}
then exists and is equal to $\inf_{n}\ncal(\tau,R;n)^\frac{1}{n}$.
Note that we are justified in dropping $R$ from the notation $\ncal(\tau)$ since it does not depend on $R$ in view of
\eqref{eq:ncal_submultplicative}. 
\begin{dfn}
We say that $f=f_{(E,\tau)}$ is partially captive if  $\ncal(\tau)=1$.
\end{dfn}
This condition is equivalent to the condition introduced and used in \cite{Faure,nw}. (See also \cite{Tsujii}. For completeness, the equivalence is demonstrated in Appendix \ref{app:PC}.)
However, the  partial captivity condition is not proved to be generic in \cite{Faure,nw} and nowhere else. Here we provide the required:
\begin{thm}\label{conj:maingoal}
Let $r\ge 2$ and suppose that the expanding map $E:\Sone\to\Sone$ is fixed. For every $\varrho>1$, there is an open dense subset 
$\mathscr V\subset \Ci^r(\Sone)$ such that 
if $\tau\in \mathscr V$ then 
\begin{equation}\label{eq:genericconditionvariant}
\ncal(\tau)<\varrho.
\end{equation}
Consequently, there is a residual subset 
$\mathscr R\subset \Ci^r(\Sone)$ such that 
$\ncal(\tau)=1$ for $\tau\in \mathscr R$. 
\end{thm}

\section{Proof of Theorem \ref{conj:maingoal}}

We henceforth fix $R$ and sometimes drop it from the notation.

\subsection{Notation}

Let $\mathcal A =\{ 0,\ldots ,\ell-1\}$. 
We suppose that $0\in \Sone$ is one of the fixed points of $E$.
Let $\mathcal I(j)=[\alpha_j,\beta_j)\subset \Sone$, $j\in \mathcal{A}$, be the semi-open intervals obtained by dividing $\Sone$ at the $\ell$ points in $E^{-1}(0)$, so that $E$ maps each of them onto $\Sone$ bijectively. 
For a word $\alpha =(\alpha _n,\ldots ,\alpha _2,\alpha _1) \in \mathcal A^n$ with alphabet $\mathcal{A}$ and of length $|\alpha|:=n$, let $\mathcal I (\alpha)$ be the interval defined by
\begin{equation*}
\mathcal I(\alpha ) = \bigcap ^{n-1} _{j=0} E^{-j} (\mathcal I (\alpha _{n-j})) .
\end{equation*}
Clearly $E^n$ maps each $\mathcal I(\alpha )$ with $|\alpha|=n$ onto $\Sone$ bijectively. 
 For each $1\leq p\leq n$, we define $[\alpha ]_p\in \mathcal A^p$ to be the truncation $(\alpha _p,\ldots ,\alpha _1)$. For each $x\in \Sone$ and $\alpha\in \mathcal{A}^n$, we write $x_{\alpha}$ for the unique point in $\mathcal I(\alpha ) $ that is mapped to $x$ by $E^n$. 
 The differential $Df(z)$ actually depends only on the first component $x$ of $z=(x,s)\in \Torus$.
Thus we will sometimes write $Df(x)$ by abuse of notation. We also consistently identify the pre-images $f^{-n}(z)$ and $E^{-n}(x)$ for $z=(x,s)$.

\subsection{A difficulty caused by the nonlineairity of $E$}
To prove Theorem \ref{conj:maingoal}, we resolve by perturbations  
the situation where a unit vector $v_0\in \mathbb{R}^2$ is contained in many of  the cones in (\ref{eq:imagecones}). 
To this end we note that, because of the nonlinearity of $E$, there is much variation in the angles of the cones 
$Df^n(x_{\alpha}) \mathscr{K}_R$ for $\alpha\in \mathcal{A}^n$. As one can easily imagine, 
the larger the angle of $Df^n(x_{\alpha}) \mathscr{K}_R$, the more perturbation is needed to resolve the situation 
$v_0\in Df^n(x_{\alpha}) \mathscr{K}_R$ for some fixed $v_0$. This prevents us from applying the argument in \cite{Tsujii} literally. 
But this difficulty is compensated by the fact that
there are relatively few $\alpha\in \mathcal{A}^n$ for which $Df^n(x_{\alpha}) \mathscr{K}_R$ is larger than average. More precisely, we have the following lemma. Note that the angle of $Df^n(x_{\alpha}) \mathscr{K}_R$ is proportional to the reciprocal of $(E^n)'(x_{\alpha})$.

\begin{lem}\label{lm:distortion} There exists a constant $C>0$ such that, for any $x\in \Sone$ and $b>0$, 
\[
\#\{ y\in E^{-n}(x)\mid (E^n)'(y) \le e^{bn}  \} \le C e^{b n}. 
\]
\end{lem} 
\begin{proof} By a simple distortion argument along the backward orbits, we have that  
\begin{equation}\label{eq:distortion}
C^{-1}\le \sum_{y\in E^{-n}(x)} (E^n)'(y)^{-1}\le C\qquad\text{for each $x\in \Sone$}
\end{equation}
for a constant $C\ge 1$. Then the lemma is a direct consequence. 
\end{proof}

\subsection{Consequences of  the condition $\ncal(\tau) \ge e^{\rho}$} 
\label{ss:consequence}
Following the argument
in Tsujii~\cite{Tsujii}, we begin by analyzing the situation where condition \eqref{eq:genericconditionvariant} does not hold, that is to say, when $\ncal(\tau) \ge e^{\rho}$ for some $\rho>0$. 
To proceed with the idea described in the previous subsection, we cover the interval $[\log\lambda, \log \Lambda]$ with open intervals $I_{j}=(a_j,b_j)$, $1\le j\le J$, such that $|I_j|:=|b_j-a_j|< \rho/3$. 
We set
\[
N=N(\rho):=\left\lceil 6\rho^{-1}\log\lceil 2 \Lambda \rceil\right\rceil
\]
where $\lceil t\rceil$ denotes the smallest integer $\ge t$ for $t\in\R$.
Then we choose an integer $q$ so large that 
\[
(q+1) N \cdot e^{-q\rho/2}<1/(4J). 
\]
Let $e^{nI}=[e^{a n},e^{bn}]$ for $I=[a,b]$.

\begin{prop}\label{prop:3.3}
If $\ncal(\tau) \ge e^{\rho}$, then we can find an arbitrarily large $n$ and 
\begin{itemize}
\item a point $z_{0}=(x,s)\in \Torus$, 
\item a unit vector $v_0\in \mathbb{R}^2$,
\item an integer $1\le j\le J$,  
\item a subset $B\subset \mathcal{A}^q$ with $\#B=2(q+1) N$, 
\item subsets
$\Sigma(\beta)\subset \mathcal{A}^n$ for each $\beta\in B$ with $\#\Sigma(\beta)\ge  e^{\rho n} \cdot \ell^{-q}/(2J)$
\end{itemize}
 such that 
\begin{itemize}
\item $[\alpha]_q=\beta$ for $\alpha\in \Sigma(\beta)$,
\item $v_0\in Df^n(x_\alpha) \mathscr{K}_R$ and $ (E^n)'(x_\alpha)\in e^{n I_j }$ for $\alpha\in \Sigma(\beta)$ with $\beta\in B$. 
\end{itemize}
\end{prop}

\begin{proof}
Since $\ncal(\tau)\ge e^\rho$ we can find an arbitrarily large $n$ such that 
\begin{equation*}
\ncal(\tau;n)> e^{(\rho/2)q}\cdot  \ncal(\tau;n-q).
\end{equation*}
By definition, we can find a point $z_0=(x,s)$ and a unit vector $v_0\in \mathbb{R}^2$ such that 
\begin{equation*}
\#\{ \alpha\in \mathcal{A}^n\mid v_0\in Df^n(x_\alpha) \mathscr K_R \} =\ncal(\tau;n).
\end{equation*}
Hence, there is a $1\le j\le J$ such that the set 
\[
B'=\{ \alpha\in \mathcal{A}^n\mid v_0\in Df^n(x_\alpha) \mathscr K_R, \ (E^n)'(x_\alpha)\in e^{n I_j} \}
\]
satisfies $\# B'\ge \ncal(\tau;n)/J$.
We divide the set $B'$ into
\[
\Sigma(\beta)=\{ \alpha\in B' \mid [\alpha]_q=\beta\}\qquad\text{for $\beta\in \mathcal{A}^q$.}
\]
Note that, for $\alpha\in \mathcal{A}^n$ with $[\alpha]_q=\beta$, we have $v_0\in Df^n(x_\alpha) \mathscr K_R$ if and only if $(Df^q(x_{\beta}))^{-1}(v_0)\in Df^{n-q}(x_{\alpha}) \mathscr K_R$. This implies
\begin{equation*}
\#\Sigma(\beta)\le \ncal(\tau;n-q)\le e^{-(\rho/2)q}\cdot \ncal(\tau;n).
\end{equation*}
Note also that we obviously have
\[
\sum_{\beta\in \mathcal{A}^q}\#\Sigma(\beta)=\#B'\ge  \ncal(\tau;n)/J.
\]
We now pick the $2(q+1) N$ largest sets among $\Sigma(\beta)$ for $\beta\in \mathcal A^q$, and let $B\subset \mathcal{A}^q$ be the corresponding $2(q+1) N$ elements in $\mathcal{A}^q$.
By the estimates above we have
\[
\#\Sigma(\beta)\ge \ncal(\tau;n)\cdot \ell^{-q}/(2J)\quad\text{for $\beta\in B$},
\]  
because the average of $\#\Sigma(\beta)$ over the rest $\mathcal{A}^q\setminus B$ must be bounded from below by
\[
\ell^{-q}\left(\ncal(\tau;n)/J-e^{-(\rho/2)q}\cdot \ncal(\tau;n)\cdot  2(q+1) N\right)
\ge \ncal(\tau;n)\cdot \ell^{-q}/(2J)
\]
where the last inequality follows from the choice of $q$.
Since $\ncal(\tau;n)$ is sub-multi\-plicative, we have $\ncal(\tau;n)\ge e^{\rho n}$ for all $n\ge 1$, which completes the proof.
\end{proof}

Below we rewrite the conclusion of Proposition \ref{prop:3.3} in a form that is more suitable for the perturbation argument in the next subsection. By choosing $\varepsilon>0$ so small that the intervals $I_j^\prime=(a_j+\varepsilon, b_j-\varepsilon)$ for $1\le j\le J$ still cover $[\log \lambda, \log \Lambda]$, we may assume that the conclusion of Proposition \ref{prop:3.3} holds with $I_j$ replaced by $I_j^\prime\Subset I_j$. Next, we rewrite the condition $v_0\in Df^n(x_\alpha) \mathscr{K}_R$.
Let us define 
\begin{equation}\label{eq:solutiontocohomologicaltruncated}
S_n(x;\alpha)=\sum_{k=1}^n \frac{\tau'(x_{[\alpha]_k})}{(E^{k})'(x_{[\alpha]_k})}.
\end{equation}
This is nothing but the slope of the image of the horizontal line by $Df^{n}(x_{\alpha})$.
In particular, $v_0\in Df^n(x_\alpha) \mathscr{K}_R$ 
is equivalent to $\lvert S_n(x;\alpha)-S\rvert\le \vartheta_R\cdot ((E^n)'(x_\alpha))^{-1}$, where $S$ denotes the slope of $v_0$. Hence, the conditions 
$v_0\in Df^n(x_\alpha) \mathscr{K}_R$ and $(E^n)'(x_\alpha)\in e^{nI_j'}$ imply that
\begin{equation*}
\lvert S_n(x; \alpha) - S\rvert \leq \vartheta_R \cdot 
((E^{n})'(x_\alpha))^{-1}\le \vartheta_R \cdot e^{-(a_j+\varepsilon)n}.
\end{equation*}
Finally, we shift the first component $x$ of the point $z_0=(x,s)$ and the unit vector $v_0$ respectively to nearby points in the finite sets
\[
T(n)=\{x\in \Sone\mid \lceil 2\Lambda\rceil^nx\in \mathbb{Z}\}
\] 
and
\[
S(n)=\{(\cos 2\pi \theta, \sin 2\pi \theta) \in \Sone\mid \lceil 2\Lambda\rceil^n \theta\in \mathbb{Z}\}.
\]
This will change the values of $S_n(x; \alpha)$, $(E^n)'(x_\alpha)$ and the slope of $v_0$ slightly. However, it is easy to see that the grids $T(n)$ and $S(n)$ are fine enough for us to conclude that, after the shift, $(E^n)'(x_\alpha)$ belongs to the larger interval $e^{nI_j}$, and that $\lvert S_n(x; \alpha) -S \rvert \leq e^{-a_j n} \cdot \vartheta_R $, where $S$ is the slope of the shifted vector $v$. Hence, we can replace the parts 
\begin{center}
``a point $z_{0}\in \Torus$'', \quad 
 ``a unit vector $v_0\in \mathbb{R}^2$'',\quad ``$v_0\in Df^n(x_\alpha) \mathscr{K}_R$''
\end{center}
in the conclusion of Proposition \ref{prop:3.3} with
\begin{center}
``a point $x\in T(n)$'', \quad 
 ``a unit vector $v_0=(\cos \theta,\sin \theta)\in S(n)$'',\\ ``$\lvert S_n(x; \alpha) -\tan\theta \rvert \leq e^{-a_j n} \cdot \vartheta_R $''
\end{center}
respectively. 

\subsection{Generic perturbations}  
Let $\tau \in \mathscr C^r (\Sone)$. For perturbations of $\tau$, we will take a set of functions $\varphi _i \in \mathscr C^r (\Sone)$, $1\leq i\leq m$, and consider the family
\begin{equation*}
\tau _{\mathbf{t}} (x)=\tau (x) +\sum _{i=1}^mt_i \varphi _i(x)
\end{equation*}
parametrized by $\mathbf{t}=(t_i)_{i=1}^m\in \R^m$. 
For a point $x\in \Sone$, an integer $n \geq 1$ and a finite subset $A\subset \mathcal A^{n}$ with $\#A=p$, let $G_{x,A}:\R^m\to\R^p$ be the affine map defined by 
\begin{equation}\label{eq:defofG_x,A}
G_{x,A}(\mathbf{t})=\left(S_{n}(x;\alpha;\tau _{\mathbf{t}})\right)_{\alpha\in A}
\end{equation}
where we used the notation $S_{n}(x;\alpha;\tau )=S_{n}(x;\alpha)$ for the functions given by \eqref{eq:solutiontocohomologicaltruncated} to indicate the dependence on $\tau $. 
For an affine map $M:E\to F$ between Euclidean spaces, let $\Jac{(M)}$ be the modulus of the Jacobian determinant of $DM\vert _{\ker (DM)^\perp}$, the restriction of the linear part $DM$ to the orthogonal complement of its kernel when $M$ is surjective, and put $\Jac{(M)}=0$ otherwise. Obviously we have $\Jac{(M)}\ge \Jac{(M|_L)}$ for any subspace $L\subset E$. Also we have
\begin{equation}\label{eq:LebmeasureEF}
\Leb_E \{ z\in E\mid \|z\|\le 1, M(z)\in X\} \le C\Jac(M)^{-1}\Leb_{F}(X)
\end{equation}
where the constant $C$ depends only on the dimension of $E$. The next lemma is a slight generalization of \cite[Proposition 3.4]{Tsujii}.

\begin{prop}\label{prop:3.4}
For any integers $p\ge 1$ and $\nu\ge 1$, there are functions $\varphi _i \in \Ci^\infty(\Sone)$, $1\le i\le m$, such that the following holds true: For any $x\in \Sone$ and subset $B \subset \mathcal A^{\nu}$ with $\#B\ge p(\nu+1)$, there is a subset $B'\subset B$ with $\# B'=p$ such that we have
\[
\Jac{(G_{x,A})} \geq 1
\]
provided that $A$ is a subset of $\mathcal{A}^n$ with $n\ge \nu$ such that the map $\alpha\in A\mapsto [\alpha]_\nu\in B'$ is a one-to-one correspondence.
\end{prop}
\begin{remark}
We omit the proof of Proposition \ref{prop:3.4}, because it is obtained by translating  that of \cite[Proposition 3.4]{Tsujii} almost literally. 
Note that the map $E$ is now a general expanding map though it was linear $x\mapsto \ell x$ in \cite[Proposition 3.4]{Tsujii} and that we have to replace the factor $\ell^k$ in some places by (a $\Ci^1$-approximation\footnote{Since the derivative of $E^k$ is $\Ci^{r-1}$, we actually have to approximate it by a $\Ci^\infty$ function.} of) the derivative of
$E^k$. (Also note that $x_\alpha$ was denoted $\alpha(x)$ in \cite{Tsujii}.)
\end{remark}

\subsection{The end of the proof}\label{sub:s1}\label{subsection:endproof}
We take an arbitrary bounded open subset $D$ of $\Ci^r (\Sone)$. 
Then we take the constant $R>0$ such that 
$R>\lVert\tau'\rVert_\infty$ uniformly for $\tau\in D$. 
We also take $\rho>0$ arbitrarily and let $X_\rho$ be the set of $\tau\in D$ for which $\ncal(\tau)\ge  e^\rho$. For the proof of Theorem \ref{conj:maingoal}, it is enough to show that the complement of $X_\rho$ in $D$ contains a dense subset because the condition $\ncal(\tau)<e^\rho$ is an open condition on $\tau$. 
(Recall that $\ncal(\tau,R;n)$ is sub-multiplicative.) 
Let $X(n)\subset D$ be the set of $\tau\in D$ for which the conditions in the statement of  Proposition \ref{prop:3.3} (modified as described at the end of Subsection \ref{ss:consequence}) hold for $n$. 
Suppose $\tau\in D$. To define the family $\tau_\mathbf{t}$, we take the functions $\varphi_i\in \mathscr C^\infty (\Sone)$, $1\le i\le m$, in Proposition \ref{prop:3.4} by choosing $p=N$ and $\nu=q$. 
Theorem \ref{conj:maingoal} follows from the next proposition.

\begin{prop}
$\displaystyle 
\Leb_{\mathbb{R}^m}\{ \mathbf{t}\in [-1,1]^m \mid \tau_{\mathbf{t}}\in X(n)\text{ for infinitely many $n$}\} =0.$
\end{prop}
\begin{proof} 
We focus on how quantities below depend on $n$ and use $C$ to denote generic constants which do not depend on $n$. 
For given $n$, a point $x\in T(n)$,  a vector $(\cos\theta,\sin\theta)\in S(n)$,  an integer $1\le j\le J$ and a subset $B\subset \mathcal{A}^q$ with $\#B=2(q+1) N$, let $X(n; x,\theta,j,B)$ be the set of $\tau\in D$ for which the conditions in the (modified) statement of  Proposition \ref{prop:3.3} hold.
Let $B'\subset B$ with $\# B'=N$ be the subset in Proposition \ref{prop:3.4}. (We suppose that $B'$ is chosen uniquely for each $B$.) 
Let $\Sigma(\beta)\subset \mathcal{A}^n$ be those in Proposition \ref{prop:3.3}, but considered now only for $\beta\in B'$.
From the choice of the functions $\varphi_i\in \mathscr C^\infty (\Sone)$, 
we have
\[
\Leb\{ \mathbf{t}\in [-1,1]^m\mid  \lvert S_n(x; \alpha_i;\tau_\mathbf{t}) -\tan\theta \rvert \leq e^{-a_j n}\cdot \vartheta_R\}\le C e^{-a_j n N}
\] 
for any combination $(\alpha_1,\alpha_2, \cdots, \alpha_N)$ of elements in $\mathcal{A}^n$ such that
\begin{itemize}
\item  $(E^n)'(x_{\alpha_i})\in e^{nI_j}=[e^{a_j n},e^{b_j n}]$, 
\item  $[\alpha_i]_{q}$, $1\le i\le N$, are in one-to-one correspondence with the elements of $B'$.
\end{itemize}
(This is simply a result of \eqref{eq:LebmeasureEF} applied to a translate of \eqref{eq:defofG_x,A}.)
On the one hand, the number of possible such combinations $(\alpha_1,\alpha_2, \cdots, \alpha_N)$ is at most $C e^{b_j nN}$ by Lemma \ref{lm:distortion}.  On the other hand,  $\tau_\mathbf{t}$ belongs to $X(n; x,\theta,j,B)$ only if the condition 
\[
\lvert S_n(x; \alpha_i;\tau_\mathbf{t}) -\tan\theta \rvert \leq e^{-a_j n}\cdot \vartheta_R
\]
holds for at least $(\# \Sigma(\beta))^N$ of these possible combinations. 
Therefore we have 
\begin{align*}
\Leb\{ \mathbf{t}\in [-1,1]^m\mid \tau_{\mathbf{t}}\in  X(n; x,\theta,j,B) \}&
\le 
\frac{C e^{-a_j n N}\cdot C e^{b_j nN}}{(\# \Sigma(\beta))^N }
\\
&\le 
\frac{C e^{(b_j-a_j) n N}}{ (e^{\rho n} \cdot \ell^{-q}/(2J))^{N}}.
\end{align*}
Taking the number of possible choices for $x$, $\theta$, $j$ and $B$ into account, we obtain
\[
\Leb\{ \mathbf{t}\in [-1,1]^m \mid \tau_{\mathbf{t}}\in  X(n) \}
\le 
\frac{C e^{(b_j-a_j) n N}\cdot e^{2n \log \lceil 2\Lambda\rceil}}{e^{\rho nN}}\le Ce^{-(\rho/3)nN}
\]
where the latter inequality is a consequence of the choice of $N$ and the intervals $I_j$.
Therefore the conclusion follows by the Borel-Cantelli lemma. 
\end{proof}

\appendix

\section{The definition of partial captivity}\label{app:PC}

Here we show that our definition of partial captivity indeed coincides with the one introduced by Faure~\cite[Definition 15]{Faure}. We fix $\tau$ and mostly suppress it from the notation. For $\alpha\in\mathcal A^\infty$, let
\[
S(x;\alpha)=\sum_{k=1}^\infty\frac{\tau'(x_{[\alpha]_k})}{(E^k)'(x_{[\alpha]_k})}.
\]
Then $S_n(x;\alpha)$ defined by \eqref{eq:solutiontocohomologicaltruncated} is the truncation of $S(x;\alpha)$ of length $n$. 
For $\widetilde R>0$, set
\[
\widetilde{\mathcal N}_{\widetilde R}(n)=\sup_{y,\eta}\#\{
\alpha\in\mathcal A^n\mid\lvert \eta-S(y;\alpha)\rvert\le\widetilde R\cdot ((E^n)'(y_\alpha))^{-1}\}.
\]
By \cite[Proposition 17]{Faure} it follows that $f=f_{E,\tau}$ is partially captive in the sense of Faure if and only if
\begin{equation}\label{eq:FaurePC}
\lim_{n\to\infty}n^{-1}\log{\widetilde{\mathcal N}_{\widetilde R}(n)}=0
\end{equation}
for all $\widetilde R>0$. (We remark that $S$ and $S_n$ appear in \cite{Faure,nw} with a change of sign.)

We first show that if $\ncal(\tau)=1$ then \eqref{eq:FaurePC} holds.
Assume therefore that $\ncal(\tau)=1$ and let $\widetilde R>0$ be arbitrary. Fix $z=(x,s)\in\T^2$ and $\eta\in\R$. Define a unit vector $v=(\cos\theta,\sin\theta)$ by setting $\theta=\arctan\eta$. Suppose now that $\alpha$ is a word in $\mathcal A^n$ such that
\[
\lvert \eta-S(x;\alpha)\rvert\le\widetilde R\cdot ((E^n)'(x_\alpha))^{-1}.
\]
Note that the set of pre-images $\zeta\in f^{-n}(z)$ is in one-to-one correspondence with $\alpha\in\mathcal A^n$ via $\pi_x(\zeta)=x_\alpha$, where $\pi_x:\T^2\to\Sone$ is the natural projection onto the first coordinate. A simple calculation shows that
\[
\lvert S_n(x;\alpha)-\tan\theta\rvert\le(\widetilde R+\vartheta_\tau)((E^n)'(x_\alpha))^{-1}
\]
so $v\in Df^n(\zeta)\mathscr K_R$ whenever $\vartheta_R-\vartheta_\tau\ge\widetilde R$.
It follows that $\widetilde{\mathcal N}_{\widetilde R}(n)\le\ncal(\tau,R;n)$, so \eqref{eq:FaurePC} holds, as promised.

The reversed implication is proved in an analogue fashion. Given $v_0\in Df^n(\zeta)\mathscr K_R$ with $v_0=(\cos\theta_0,\sin\theta_0)$ we set $\eta_0=\tan\theta_0$. With $\alpha$ given by the correspondence $\pi_x(\zeta)=x_\alpha$, a calculation similar to the one above shows that
\[
\lvert \eta_0-S(x;\alpha)\rvert\le(\vartheta_R+\vartheta_\tau)((E^n)'(x_\alpha))^{-1}
\]
so if $\widetilde R>\vartheta_R+\vartheta_\tau$ then $\ncal(\tau,R;n)\le\widetilde{\mathcal N}_{\widetilde R}(n)$. Hence $\ncal(\tau)=1$ by virtue of \eqref{eq:FaurePC}.

\section{The transversality condition}\label{app:B}
In this appendix, we briefly discuss two other quantities defined similarly to $\ncal(\tau)$. 
We define
\begin{equation*}
\mbold(\tau,R;n)=\sup_z{\sup_{w\in f ^{-n}(z)}{\sum_{\zeta\not\pitchfork w}
\frac{1}{\det Df ^n(\zeta)}}}
\end{equation*}
where the sum is taken over $\zeta\in f ^{-n}(z)$ such that 
$
Df ^n(\zeta) \mathscr K_R\cap Df ^n(w)\mathscr K_R=\{0\}$.  We then set
\[
\mbold(\tau):=\limsup_{n\to\infty}\mbold(\tau,R;n)^\frac{1}{n}
\]
which does not depend on the choice of $R$. 
Similarly, for $z\in\Torus$ and $n\ge 2$, we set
\begin{equation*}
\nbold(\tau,R;n)=\sup_z{\sup_v{\sum_{\substack{\zeta\in f^{-n}(z)\\
v\in Df^n(\zeta)\mathscr K_R}} \frac{1}{\det Df ^n(\zeta)}}},
\end{equation*}
where the supremum $\sup_v$ is taken over all unit vectors $v$ in $\mathbb{R}^2$ and the summation is taken over those points $\zeta\in f^{-n}(z)$ for which $Df^n(\zeta)\mathscr K_R$ contains $v$.
This is sub-multiplicative so that we can define 
\begin{equation*}
\nbold(\tau):=\limsup_{n\to\infty}\nbold(\tau;n)^{\frac{1}{n}}.
\end{equation*}

\begin{lem} The following are all equivalent:
\begin{enumerate}[{\normalfont 1.}]
\item $\tau$ is cohomologous to a constant,
\item $\mbold(\tau)=1$,
\item $\nbold(\tau)=1$,
\item $\ncal(\tau)=\ell$.
\end{enumerate}
\end{lem}

\begin{proof}
For the equivalence of items $1$, $2$, $3$, we refer to \cite{ButterleyEslami}. It is clear that item 4 follows from item 1. 
We prove that $\nbold(\tau)<1$ implies $\ncal(\tau)<\ell$ to complete the proof. 
If $\nbold(\tau)<1$, we have 
$\nbold(\tau,R;n)<(1-\varepsilon)^n$ for sufficiently small $\varepsilon>0$ and sufficiently large $n$. 
This means that $\ncal(\tau,R;n)$ cannot be equal to $\ell^n$ for infinitely many $n$, because for each such $n$ we have $C^{-1}\le\nbold(\tau,R;n)<(1-\ve)^n$ by \eqref{eq:distortion}, which leads to a contradiction if we choose $n$ large enough.
Hence, $\ncal(\tau,R;n)<\ell^n$ for sufficiently large $n$, so
$\ncal(\tau,R;n)^{\frac{1}{n}}\le \ell-\eta$ for some small $\eta>0$. Since $\ncal(\tau,R;n)$ is sub-multiplicative, we find that 
$\ncal(\tau)\le \ell-\eta<\ell$. 
\end{proof}

From the lemma above, we see that the partial captivity condition $\ncal(\tau)=1$ is a  much stronger condition than requiring that 
$\tau$ \emph{not} be cohomologous to a constant. Nevertheless, Theorem \ref{conj:maingoal} shows that partial captivity is still a generic condition. 
Finally, setting $\chi=\lim_{n\to \infty} (\min (E^n)')^{-\frac{1}{n}}\le \lambda^{-1}<1$, we observe that
\[
\chi\le \mbold(\tau)\le 1, \quad
\chi\le \nbold(\tau)\le 1, \quad
1\le \ncal(\tau)\le \ell
\]
in general. The partial captivity condition $\ncal(\tau)=1$ implies that $\mbold(\tau)$ and $\nbold(\tau)$ take the smallest possible value, that is, 
$
\mbold(\tau)=\nbold(\tau)=\chi$.

\section*{Acknowledgements}
The research of Jens Wittsten was supported by Knut och Alice Wallenbergs Stiftelse Grant No.~2013.0347.
The research of Masato Tsujii was supported by JSPS Kakenhi Grant No.~20251598.


\bibliographystyle{plain}
\bibliography{Refs}

\begin{thebibliography}{1}

\bibitem{ButterleyEslami}
Oliver Butterley and Peyman Eslami.
\newblock Exponential mixing for skew products with discontinuities.
\newblock To appear in {\it Trans.~Amer.~Math.~Soc.},
  arXiv:1405.7008\setbox0=\hbox{2014}.

\bibitem{Dolgopyat2002}
Dmitry Dolgopyat.
\newblock {On mixing properties of compact group extensions of hyperbolic
  systems}.
\newblock {\em Israel J. Math.}, 130:157--205, 2002.

\bibitem{Faure}
Fr{\'e}d{\'e}ric Faure.
\newblock Semiclassical origin of the spectral gap for transfer operators of a
  partially expanding map.
\newblock {\em Nonlinearity}, 24:1473--1498, 2011.

\bibitem{nw}
Yushi Nakano and Jens Wittsten.
\newblock {On the specrtra of quenched random perturbations of partially
  expanding maps on the torus}.
\newblock {\em Nonlinearity}, 28:951--1002, 2015.

\bibitem{Tsujii}
Masato Tsujii.
\newblock Decay of correlations in suspension semi-flows of angle-multiplying
  maps.
\newblock {\em Ergodic Theory Dyn. Syst.}, 28:291--317, 2008.

\end{thebibliography}

\end{document}